\documentclass[13pt,reqno]{amsart} 

\usepackage{amssymb,latexsym}
\usepackage[draft=true]{hyperref}
\usepackage{cite} 
\usepackage{epsfig}
\usepackage[height=190mm,width=130mm]{geometry} 

\theoremstyle{plain}
\newtheorem{theorem}{Theorem}
\newtheorem{lemma}{Lemma}
\newtheorem{corollary}{Corollary}

\theoremstyle{definition}

\theoremstyle{remark}
\newtheorem{remark}{Remark}

\numberwithin{equation}{section} 

\begin{document}
\title[On a system of difference equations ]{On a system of difference equations of third order solved in closed form} 

\author{Y. Akrour}
\address{Youssouf Akrour, LMAM Laboratory, Department of Mathematics, Mohamed Seddik Ben Yahia University-Jijel and Department des Sciences Exactes et d'Informatique, Ecole Normale Sup¨¦rieure-Constantine, Algeria.}
\email{youssouf.akrour@gmail.com}

\author{N. Touafek}
\address{Nouressadat Touafek, LMAM Laboratory, Department of Mathematics \\University of Mohamed Seddik Ben Yahia-Jijel \\ Algeria.}
\email{ntouafek@gmail.com}

\author{Y. Halim}
\address{Yacine Halim, Mila University Center\\ Department of Mathematics and Computer Science and LMAM Laboratory, University of Mohamed Seddik Ben Yahia-Jijel\\ Algeria.}
\email{halyacine@yahoo.fr}

\begin{abstract}
In this note we show the the system of difference equations
$$ x_{n+1}=\dfrac{ay_{n-2}x_{n-1}y_n+bx_{n-1}y_{n-2}+cy_{n-2}+d}{y_{n-2}x_{n-1}y_n},$$
$$y_{n+1}=\dfrac{ax_{n-2}y_{n-1}x_n+by_{n-1}x_{n-2}+cx_{n-2}+d}{x_{n-2}y_{n-1}x_n},$$
where $n\in \mathbb{N}_{0}$, the initial values $x_{-2}$, $x_{-1}$, $x_0$, $y_{-2}$, $y_{-1}$ and $y_0$ are arbitrary nonzero real numbers and the parameters $a$, $b$, $c$ and $d$ are arbitrary real numbers with $d\ne 0$, can be solved in a closed form.\\ We will see that when $a=b=c=d=1$ the solutions are expressed using the famous Teteranacci numbers.  In particular, the results obtained here extend those in our work \cite{arxiv}.
\end{abstract}




\maketitle

\section{Introduction}

Solving difference equations and their systems is a subject that attract the attention of several researchers and a big number of papers is devoted to this line of research where various models are proposed. We can consult for example the papers \cite{arxiv}-\cite{Yazlik Tollu Taskara:2013}, for some concrete  models of such equations and systems where also we can understands  the techniques used in solving these models.

Recently in \cite{arxiv} and as generalization of the equations and systems studied in \cite{azizi}, \cite{Halim Rabago:2018}, \cite{stevicejqtde2018}, \cite{Yazlik Tollu Taskara:2013}, we have solved in a closed form the system of difference equations
\begin{eqnarray}\label{system-tribonacci}
\begin{cases}
 x_{n+1}=\dfrac{ay_{n}x_{n-1}y_n+bx_{n-1}+c}{x_{n-1}y_n},\\
y_{n+1}=\dfrac{ax_{n}y_{n-1}x_n+by_{n-1}+c}{y_{n-1}x_n},
\end{cases}
\end{eqnarray}
Here and motivated by the above mentioned papers we show that we are able to expressed in a closed form the well defined solutions of following system of difference equations
\begin{eqnarray}\label{system-generalized}
\begin{cases}
 x_{n+1}=\dfrac{ay_{n-2}x_{n-1}y_n+bx_{n-1}y_{n-2}+cy_{n-2}+d}{y_{n-2}x_{n-1}y_n}, \\
y_{n+1}=\dfrac{ax_{n-2}y_{n-1}x_n+by_{n-1}x_{n-2}+cx_{n-2}+d}{x_{n-2}y_{n-1}x_n},
\end{cases}
\end{eqnarray}
where $n\in \mathbb{N}_{0}$, the initial values $x_{-2}$, $x_{-1}$, $x_0$, $y_{-2}$, $y_{-1}$ and $y_0$ are arbitrary nonzero real numbers and the parameters $a$, $b$, $c$ and $d$ are arbitrary real numbers with $d\ne 0$.

Clearly if $d=0$, then system \eqref{system-generalized} is nothing other than system \eqref{system-tribonacci}.  For the readers interested in the solutions of this system, we refer to \cite{arxiv}, where the system \eqref{system-tribonacci} was been completely solved.\\ Noting also that the system \eqref{system-generalized} can be seen as a generalization of the equation
\begin{equation}\label{azizieq}
 x_{n+1}=\dfrac{ax_{n-2}x_{n-1}x_n+bx_{n-1}x_{n-2}+cx_{n-2}+d}{x_{n-2}x_{n-1}x_n},\,n\in \mathbb{N}_{0}.
 \end{equation}
In fact the solutions of \eqref{azizieq} can be obtained from the solutions of \eqref{system-generalized} by choosing $y_{-i}=x_{-i},\,i=0,1,2$. The equation \eqref{azizieq} was been the subject of substantial part of the paper of Azizi \cite{azizi1}, which also motivated our present study.

We will see that the explicit formulas of the well defined solutions of system \eqref{system-generalized} are expressed using the terms of the sequence  $\left(J_{n}\right)_{n=0}^{+\infty}$ which are the solutions of the fourth order linear homogeneous difference equation defined by the relation
\begin{equation}\label{tribonacci_sequence_J(n)}
J_{n+4}=aJ_{n+3}+bJ_{n+2}+cJ_{n+1}+dJ_{n},\quad n\in \mathbb{N}_{0},
\end{equation}
and the special initial values
\begin{equation}
J_0=0, \quad J_1=0, \quad J_2=1 \; \text{and} \; J_3=a.
\end{equation}
Now we solve in closed form the equation \eqref{tribonacci_sequence_J(n)}. This equation (with the same or different initial values and parameters) was the subject of some papers in the literature \cite{waddill1992tetranacci}, \cite{hathiwala2017binet}.

The characteristic equation associated to the equation \eqref{tribonacci_sequence_J(n)} is

 \begin{equation}
\lambda^4-a\lambda^3-b\lambda^2-c\lambda-d=0   \label{eq_charac_of_GTS}
\end{equation}

and let $\alpha$, $\beta$, $\gamma$ and $\delta$  its four roots, then
\begin{equation} \label{relation racines}
\begin{cases}
\alpha + \beta + \gamma+\delta=a \\
\alpha \beta+\alpha \gamma+\alpha\delta+\beta \gamma+\beta \delta+\gamma \delta =-b \\
\alpha \beta \gamma+\alpha \beta \delta+\alpha \gamma \delta+\beta\gamma\delta =c\\
\alpha \beta \gamma \delta=-d
\end{cases}
\end{equation}

We have:

\textbf{Case 1: If all roots are real and equal.} In this case
\begin{equation*}
J_n=\left(  c_{1}+c_{2}n+c_{3}n^2+c_{4}n^3 \right) \alpha^n.
\end{equation*}
Now using \eqref{relation racines} and the fact that $J_{0}=0$, $J_{1}=0$, $J_{2}=1$ and $J_{3}=a$, we obtain
\begin{equation}  \label{binet_formula1_of_J(n)}
J_n=\left( \dfrac{-n+n^3}{6\alpha^2}\right) \alpha^n
\end{equation}

\textbf{Case 2: If three roots are real and equal, say $\beta=\gamma=\delta$.} In this case
\begin{equation*}
J_n=c_{1} \alpha^n+\left( c_2 + c_{3}n+c_{4}n^2 \right)\beta^n.
\end{equation*}
Now using \eqref{relation racines} and the fact that $J_{0}=0$, $J_{1}=0$, $J_{2}=1$ and $J_{3}=a$, we obtain
\begin{equation} \label{binet_formula2_of_J(n)}
J_n=\dfrac{-\alpha}{(\beta-\alpha)^3} \alpha^n+\left( \dfrac{\alpha}{(\beta-\alpha)^3} - \dfrac{n(\alpha+\beta)}{2\beta
(\beta-\alpha)^2}+\dfrac{n^2}{2\beta(\beta-\alpha)}\right)\beta^n,
\end{equation}

\textbf{Case 3: If two real roots are equal, say $\gamma=\delta$.} In this case
\begin{equation*}
J_n=c_{1} \alpha^n+c_2 \beta^n+ \left(c_{3}+c_{4}n\right)\gamma^n.
\end{equation*}
Now using \eqref{relation racines} and the fact that $J_{0}=0$, $J_{1}=0$, $J_{2}=1$ and $J_{3}=a$, we obtain
\begin{equation} \label{binet_formula3_of_J(n)}
J_n=\dfrac{-\alpha}{(\gamma-\alpha)^2(\beta-\alpha)} \alpha^n+\dfrac{\beta}{(\gamma-\beta)^2(\beta-\alpha)}\beta^n +\left(  \dfrac{\alpha \beta-\gamma^2}{(\gamma-\alpha)^2(\gamma-\beta)^2}+\dfrac{n}{(\gamma-\alpha)(\gamma-\beta)}\right)\gamma^n,
\end{equation}

\textbf{Case 4: If double two real roots are equal, say $\alpha=\beta \neq \gamma=\delta$.} In this case
\begin{equation*}
J_n=\left( c_{1} +c_{2} n \right) \alpha^n+ \left(c_{3}+c_{4} n \right)\gamma^n.
\end{equation*}
Now using \eqref{relation racines} and the fact that $J_{0}=0$, $J_{1}=0$, $J_{2}=1$ and $J_{3}=a$, we obtain
\begin{equation} \label{binet_formula4_of_J(n)}
J_n=\left(\dfrac{\gamma+\alpha}{(\gamma-\alpha)^3}+\dfrac{n}{(\gamma-\alpha)^2}\right)\alpha^n+\left(-\dfrac{\gamma+\alpha}{(\gamma-\alpha)^3}+\dfrac{n}{(\gamma-\alpha)^2}\right)\gamma^n,
\end{equation}

\textbf{Case 5: If the roots are all real and different.} In this case
\begin{equation*}
J_n=c_{1} \alpha^n+c_{2}\beta^n+c_{3}\gamma^{n}+c_{4}\delta^n.
\end{equation*}
Again, using \eqref{relation racines} and the fact that $J_{0}=0$, $J_{1}=0$,  $J_{2}=1$ and $J_{3}=a$, we obtain
\begin{eqnarray}\label{binet_formula5_of_J(n)}
J_n&=&\dfrac{-\alpha}{(\delta-\alpha)(\gamma-\alpha)(\beta-\alpha)} \alpha^n +\dfrac{\beta}{(\delta-\beta)(\gamma-\beta)(\beta-\alpha)} \beta^n +\dfrac{-\gamma}{(\delta-\gamma)(\gamma-\beta)(\gamma-\alpha)} \gamma^n \nonumber \\
&&+\dfrac{\delta}{(\delta-\gamma)(\delta-\beta)(\delta-\alpha)} \delta^n
\end{eqnarray}

\textbf{Case 6: If two real roots are equal, say $\alpha=\beta$ and two roots are complex conjugate ones, say $\delta=\overline{\gamma}$.} In this case
\begin{equation*}
J_n=(c_{1}+c_{2} n) \alpha^n+c_{3}\gamma^{n}+c_{4}\overline{\gamma}^n.
\end{equation*}
Again, using \eqref{relation racines} and the fact that $J_{0}=0$, $J_{1}=0$,  $J_{2}=1$ and $J_{3}=a$, we obtain
\begin{eqnarray}\label{binet_formula6_of_J(n)}
J_n&=&\left(\dfrac{\overline{\gamma}\gamma-\alpha^2}{(\overline{\gamma}-\alpha)^2(\gamma-\alpha)^2} +\dfrac{n}{(\overline{\gamma}-\alpha)(\gamma-\alpha)} \right) \alpha^n+\dfrac{-\gamma}{(\overline{\gamma} -\gamma)(\gamma-\alpha)^2} \gamma^n \nonumber \\
&&+\dfrac{\overline{\gamma}}{(\overline{\gamma}-\gamma)(\overline{\gamma}-\alpha)^2} \overline{\gamma}^n
\end{eqnarray}

\textbf{Case 7: If two  real roots $\alpha$, $\beta$ are different and two roots are complex conjugate ones, say $\delta=\overline{\gamma}$.} In this case
\begin{equation*}
J_n=c_{1} \alpha^n +c_{2} \beta^n+c_{3}\gamma^{n}+c_{4}\overline{\gamma}^n.
\end{equation*}
Again, using \eqref{relation racines} and the fact that $J_{0}=0$, $J_{1}=0$,  $J_{2}=1$ and $J_{3}=a$, we obtain
\begin{eqnarray}\label{binet_formula7_of_J(n)}
J_n&=&\dfrac{-\alpha}{(\overline{\gamma}-\alpha)(\gamma-\alpha)(\beta-\alpha)} \alpha^n +\dfrac{\beta}{(\overline{\gamma}-\beta)(\gamma-\beta)(\beta-\alpha)} \beta^n +\dfrac{-\gamma}{(\overline{\gamma}-\gamma)(\gamma-\beta)(\gamma-\alpha)} \gamma^n \nonumber \\
&&+\dfrac{\overline{\gamma}}{(\overline{\gamma}-\gamma)(\overline{\gamma}-\beta)(\overline{\gamma}-\alpha)} \overline{\gamma}^n
\end{eqnarray}

\textbf{Case 8: If two  complex roots are equal, say $\alpha=\gamma$ and  $\beta=\delta=\overline{\alpha}$.} In this case
\begin{equation*}
J_n=(c_{1}  +c_{2} n)\alpha^n +(c_{3}+c_{4} n )\overline{\alpha}^n.
\end{equation*}
Again, using \eqref{relation racines} and the fact that $J_{0}=0$, $J_{1}=0$,  $J_{2}=1$ and $J_{3}=a$, we obtain
\begin{equation}\label{binet_formula8_of_J(n)}
J_n=\left( \dfrac{\overline{\alpha}+ \alpha}{(\overline{\alpha}-\alpha)^3} + \dfrac{n}{(\overline{\alpha}-\alpha)^2} \right) \alpha^n + \left( \dfrac{-\overline{\alpha}- \alpha}{(\overline{\alpha}-\alpha)^3} + \dfrac{n}{(\overline{\alpha}-\alpha)^2} \right) \overline{\alpha}^n
\end{equation}

\textbf{Case 9: If the roots are all complex and different, say $\beta=\overline{\alpha}$ and $\delta=\overline{\gamma}$.} In this case
\begin{equation*}
J_n=c_{1} \alpha^n+c_{2}\overline{\alpha}^n+c_{3}\gamma^{n}+c_{4}\overline{\gamma}^n.
\end{equation*}
Again, using \eqref{relation racines} and the fact that $J_{0}=0$, $J_{1}=0$,  $J_{2}=1$ and $J_{3}=a$, we obtain
\begin{eqnarray}\label{binet_formula9_of_J(n)}
J_n&=&\dfrac{-\alpha}{(\overline{\gamma}-\alpha)(\gamma-\alpha)(\overline{\alpha}-\alpha)} \alpha^n +\dfrac{\overline{\alpha}}{(\overline{\gamma}-\overline{\alpha})(\gamma-\overline{\alpha})(\overline{\alpha}-\alpha)} \overline{\alpha}^n +\dfrac{-\gamma}{(\overline{\gamma}-\gamma)(\gamma-\overline{\alpha})(\gamma-\alpha)} \gamma^n \nonumber \\
&&+\dfrac{\overline{\gamma}}{(\overline{\gamma}-\gamma)(\overline{\gamma}-\overline{\alpha})(\overline{\gamma}-\alpha)} \overline{\gamma}^n
\end{eqnarray}

\section{The main theorem and some particular cases}
Here, we give a closed form for the well defined solutions of the system \eqref{system-generalized} with $d\neq 0$.  To this end we will use the same change of variables as in \cite{arxiv} to transform  the system \eqref{system-generalized} to a linear one and than following the same procedure as in \cite{arxiv} to obtain the closed form of the solutions. To get the solutions of the corresponding linear system we need to solve some fourth order linear difference equations. In particular, we derive from the main result (Main Theorem), for which we leave the proof to the next section,  the solutions of some particular systems and equations where their solutions are related to the famous Tetranacci numbers.\\
We recall that by a well defined solutions of system \eqref{system-generalized}, we mean a solution that satisfies $x_{n}y_{n}\neq 0,\,n\geq -2$. The set of well defined solutions is not empty, in fact it suffices to choose the initial values and the parameters $a$, $b$, $c$ and $d$ positive, to see that every solution of \eqref{system-generalized} will be well defined.
\subsection{Closed form of well defined solutions of the system \eqref{system-generalized}}

The following result give an explicit formula for well defined solutions of the system \eqref{system-generalized}.

\begin{theorem} (Main Theorem.) \label{solution_form_of_x(n)_and_y(n)}
Let $\{x_n, y_n\}_{n\geq -1}$ be a well defined solution of
\eqref{system-generalized}. Then, for $n \in \mathbb{N}_0$, we have

\begin{equation*}
x_{2n+1}=\dfrac{dJ_{2n+2}+\left( cJ_{2n+2}+dJ_{2n+1} \right) y_{-2} +\left( J_{2n+4}-aJ_{2n+3} \right) x_{-1}y_{-2} +J_{2n+3} y_{0}x_{-1}y_{-2} }{ dJ_{2n+1} +\left( cJ_{2n+1}+dJ_{2n} \right) y_{-2} +\left( J_{2n+3}-aJ_{2n+2} \right) x_{-1}y_{-2}+J_{2n+2} y_{0}x_{-1}y_{-2}} ,\\
\end{equation*}
\begin{equation*}
x_{2n+2}=\dfrac{dJ_{2n+3}+\left( cJ_{2n+3}+dJ_{2n+2} \right) x_{-2} +\left( J_{2n+5}-aJ_{2n+4} \right) y_{-1}x_{-2} +J_{2n+4}x_{0}y_{-1}x_{-2}}{dJ_{2n+2}+\left( cJ_{2n+2}+dJ_{2n+1} \right) x_{-2} +\left( J_{2n+4}-aJ_{2n+3} \right) y_{-1}x_{-2} +J_{2n+3}x_{0}y_{-1}x_{-2} },
\end{equation*}
\begin{equation*}
y_{2n+1}=\dfrac{dJ_{2n+2}+\left( cJ_{2n+2}+dJ_{2n+1} \right) x_{-2} +\left( J_{2n+4}-aJ_{2n+3} \right) y_{-1}x_{-2} +J_{2n+3}x_{0}y_{-1}x_{-2}}{dJ_{2n+1}+\left( cJ_{2n+1}+dJ_{2n} \right) x_{-2} +\left( J_{2n+3}-aJ_{2n+2} \right) y_{-1}x_{-2}+J_{2n+2} x_{0}y_{-1}x_{-2} },\\
\end{equation*}
\begin{equation*}
y_{2n+2}=\dfrac{dJ_{2n+3}+\left( cJ_{2n+3}+dJ_{2n+2} \right) y_{-2} +\left( J_{2n+5}-aJ_{2n+4} \right) x_{-1}y_{-2} +J_{2n+4} y_{0}x_{-1}y_{-2}}{dJ_{2n+2}+\left( cJ_{2n+2}+dJ_{2n+1} \right) y_{-2} +\left( J_{2n+4}-aJ_{2n+3} \right) x_{-1}y_{-2}+J_{2n+3} y_{0}x_{-1}y_{-2}}.
\end{equation*}

where the initial values $x_{-2}, x_{-1},x_{0}, y_{-2},  y_{-1}$ and $y_{0} \in
\left(\mathbb{R}-\left\{0\right\}\right)-F$, with $F$ is the Forbidden set of system
\eqref{system-generalized} given by
\begin{equation*}
F=\bigcup_{n=0}^\infty \left\{(x_{-2}, x_{-1},x_0, y_{-2}, y_{-1},
y_0)\in \left(\mathbb{R}-\left\{0\right\}\right) :A_n=0 \,\text{or} \,B_n=0\right\},
\end{equation*}
where
\begin{eqnarray*}
A_n=dJ_{n+1}+\left( cJ_{n+1}+dJ_{n} \right) y_{-2} +\left( J_{n+3}-aJ_{n+2} \right) x_{-1}y_{-2}+J_{n+2} y_{0}x_{-1}y_{-2},\\
B_n=dJ_{n+1} +\left( cJ_{n+1}+dJ_{n} \right) x_{-2} +\left( J_{n+3}-aJ_{n+2} \right) y_{-1}x_{-2}+J_{n+2} x_{0}y_{-1}x_{-2}.
\end{eqnarray*}
 \end{theorem}
\subsection{Particular cases} Now, we focus our study on some particular cases of system \eqref{system-generalized}.
\subsubsection{The solutions of the equation $ x_{n+1}=\dfrac{ax_{n-2}x_{n-1}x_n+bx_{n-1}x_{n-2}+cx_{n-2}+d}{x_{n-2}x_{n-1}x_n}$}

If we choose  $y_{-2}=x_{-2}$, $y_{-1}=x_{-1}$ and $y_0=x_0$, then system \eqref{system-generalized} is reduced to the equation
\begin{equation}\label{equation-generalized}
 x_{n+1}=\dfrac{ax_{n-2}x_{n-1}x_n+bx_{n-1}x_{n-2}+cx_{n-2}+d}{x_{n-2}x_{n-1}x_n},  \; \; n\in \mathbb{N}_0.
\end{equation}
So, it follows from the Main Theorem that

\begin{corollary} \label{solution_form_of_equation_generalized}
Let $\{ x_n \}_{n\geq -1}$ be a well defined solution of the equation \eqref{equation-generalized}. Then for $n \in \mathbb{N}_0$, we have

\begin{equation*}
x_{2n+1}=\dfrac{dJ_{2n+2}+\left( cJ_{2n+2}+dJ_{2n+1} \right) x_{-2} +\left( J_{2n+4}-aJ_{2n+3} \right) x_{-1}x_{-2} +J_{2n+3} x_{0}x_{-1}x_{-2} }{dJ_{2n+1} +\left( cJ_{2n+1}+dJ_{2n} \right) x_{-2} +\left( J_{2n+3}-aJ_{2n+2} \right) x_{-1}x_{-2}+J_{2n+2} x_{0}x_{-1}x_{-2}} ,
\end{equation*}

\begin{equation*}
x_{2n+2}=\dfrac{dJ_{2n+3}+\left( cJ_{2n+3}+dJ_{2n+2} \right) x_{-2} +\left( J_{2n+5}-aJ_{2n+4} \right) x_{-1}x_{-2} +J_{2n+4}x_{0}x_{-1}x_{-2}}{dJ_{2n+2}+\left( cJ_{2n+2}+dJ_{2n+1} \right) x_{-2} +\left( J_{2n+4}-aJ_{2n+3} \right) x_{-1}x_{-2} +J_{2n+3}x_{0}x_{-1}x_{-2} },
\end{equation*}
\end{corollary}

Noting that this equation was the subject of a substantial part of the paper of Azizi in \cite{azizi}.

\subsection{The solutions of the system \eqref{system-generalized} with $ a=b=c=d=1$}

Consider the system
\begin{eqnarray}\label{system:1}
\begin{cases}
 x_{n+1}=\dfrac{y_{n-2}x_{n-1}y_n+x_{n-1}y_{n-2}+y_{n-2}+1}{y_{n-2}x_{n-1}y_n}, \\
y_{n+1}=\dfrac{x_{n-2}y_{n-1}x_n+y_{n-1}x_{n-2}+x_{n-2}+1}{x_{n-2}y_{n-1}x_n},\; n \in \mathbb{N}_0,
\end{cases}
\end{eqnarray}
which is a is particular case of the system \eqref{system-generalized} with $a=b=c=d=1$. In this case the sequence  $\{J_n \}$ is nothing other than the sequence of Tetranacci numbers $\{T_n \}$, that is
\begin{equation*}
T_{n+4}=T_{n+3}+T_{n+2}+T_{n+1}+T_{n},\, n\in \mathbb{N}_{0},\; \text{ where } \; T_0=T_1=0, \;T_2=1 \; \text{and} \; T_3=1,
\end{equation*}
and we have
\begin{eqnarray*}
T_n&=&\dfrac{-\alpha}{(\overline{\gamma}-\alpha)(\gamma-\alpha)(\beta-\alpha)} \alpha^n +\dfrac{\beta}{(\overline{\gamma}-\beta)(\gamma-\beta)(\beta-\alpha)} \beta^n +\dfrac{-\gamma}{(\overline{\gamma}-\gamma)(\gamma-\beta)(\gamma-\alpha)} \gamma^n \nonumber \\
&&+\dfrac{\overline{\gamma}}{(\overline{\gamma}-\gamma)(\overline{\gamma}-\beta)(\overline{\gamma}-\alpha)} \overline{\gamma}^n, \qquad  n \in \mathbb{N}_0,
\end{eqnarray*}
with
$$\alpha=\dfrac{1}{4}+\dfrac{1}{2}\omega+\dfrac{1}{2}\sqrt{\dfrac{11}{4}-\omega^2+\dfrac{13}{4}\omega^{-1}},\,
\beta=\dfrac{1}{4}+\dfrac{1}{2}\omega-\dfrac{1}{2}\sqrt{\dfrac{11}{4}-\omega^2+\dfrac{13}{4}\omega^{-1}},$$
$$\gamma=\dfrac{1}{4}-\dfrac{1}{2}\omega+\dfrac{1}{2}\sqrt{\dfrac{11}{4}-\omega^2-\dfrac{13}{4}\omega^{-1}},\,
\delta=\dfrac{1}{4}-\dfrac{1}{2}\omega-\dfrac{1}{2}\sqrt{\dfrac{11}{4}-\omega^2-\dfrac{13}{4}\omega^{-1}},$$
$$\omega=\sqrt{\dfrac{11}{12}+\left(\dfrac{-65}{54}+\sqrt{\dfrac{563}{108}} \right)^{\dfrac{1}{3}}+\left(\dfrac{-65}{54}-\sqrt{\dfrac{563}{108}} \right)^{\dfrac{1}{3}}}.$$
Numerically we have $\alpha=1.927561975$, $\beta=-0.774804113$ and the two complex conjugate are
$\gamma= -0.076378931+ 0.814703647i ,\, \delta=\bar{\gamma} \; \text{with} \; i^{2}=-1$.

The one dimensional version of the system \eqref{system:1}, is the equation
\begin{equation}\label{equation_tribonacci}
 x_{n+1}=\dfrac{x_{n-2}x_{n-1}x_n+x_{n-1}x_{n-2}+x_{n-2}+1}{x_{n-2}x_{n-1}x_n}, \; \; n\in \mathbb{N}_0.
\end{equation}
The following results follows respectively from the Main Theorem.
\begin{corollary}
Let $\{x_n, y_n\}_{n\geq -1}$ be a well defined solution of
\eqref{system:1}. Then, for $n \in \mathbb{N}_0$, we have

\begin{equation*}
x_{2n+1}=\dfrac{T_{2n+2}+\left( T_{2n+2}+T_{2n+1} \right) y_{-2} +\left( T_{2n+4}-T_{2n+3} \right) x_{-1}y_{-2} +T_{2n+3} y_{0}x_{-1}y_{-2} }{ T_{2n+1} +\left( T_{2n+1}+T_{2n} \right) y_{-2} +\left( T_{2n+3}-T_{2n+2} \right) x_{-1}y_{-2}+T_{2n+2} y_{0}x_{-1}y_{-2}} ,\\
\end{equation*}
\begin{equation*}
x_{2n+2}=\dfrac{T_{2n+3}+\left( T_{2n+3}+T_{2n+2} \right) x_{-2} +\left( T_{2n+5}-T_{2n+4} \right) y_{-1}x_{-2} +T_{2n+4}x_{0}y_{-1}x_{-2}}{T_{2n+2}+\left( T_{2n+2}+T_{2n+1} \right) x_{-2} +\left( T_{2n+4}-T_{2n+3} \right) y_{-1}x_{-2} +T_{2n+3}x_{0}y_{-1}x_{-2} },
\end{equation*}
\begin{equation*}
y_{2n+1}=\dfrac{T_{2n+2}+\left( T_{2n+2}+T_{2n+1} \right) x_{-2} +\left( T_{2n+4}-T_{2n+3} \right) y_{-1}x_{-2} +T_{2n+3}x_{0}y_{-1}x_{-2}}{T_{2n+1}+\left( T_{2n+1}+T_{2n} \right) x_{-2} +\left( T_{2n+3}-T_{2n+2} \right) y_{-1}x_{-2}+T_{2n+2} x_{0}y_{-1}x_{-2} },\\
\end{equation*}
\begin{equation*}
y_{2n+2}=\dfrac{T_{2n+3}+\left( T_{2n+3}+T_{2n+2} \right) y_{-2} +\left( T_{2n+5}-T_{2n+4} \right) x_{-1}y_{-2} +T_{2n+4} y_{0}x_{-1}y_{-2}}{T_{2n+2}+\left( T_{2n+2}+T_{2n+1} \right) y_{-2} +\left( T_{2n+4}-T_{2n+3} \right) x_{-1}y_{-2}+T_{2n+3} y_{0}x_{-1}y_{-2}}.
\end{equation*}
\end{corollary}

\begin{corollary} \label{solution_form_of_equation_tribonacci}
Let $\{ x_n \}_{n\geq -1}$ be a well defined solution of the equation \eqref{equation_tribonacci}. Then for $n \in \mathbb{N}_0$, we have
\begin{equation*}
x_{2n+1}=\dfrac{T_{2n+2}+\left( T_{2n+2}+T_{2n+1} \right) x_{-2} +\left( T_{2n+4}-T_{2n+3} \right) x_{-1}x_{-2} +T_{2n+3} x_{0}x_{-1}x_{-2} }{T_{2n+1} +\left( T_{2n+1}+T_{2n} \right) x_{-2} +\left( T_{2n+3}-T_{2n+2} \right) x_{-1}x_{-2}+T_{2n+2} x_{0}x_{-1}x_{-2}} ,
\end{equation*}

\begin{equation*}
x_{2n+2}=\dfrac{T_{2n+3}+\left( T_{2n+3}+T_{2n+2} \right) x_{-2} +\left( T_{2n+5}-T_{2n+4} \right) x_{-1}x_{-2} +T_{2n+4}x_{0}x_{-1}x_{-2}}{T_{2n+2}+\left( T_{2n+2}+T_{2n+1} \right) x_{-2} +\left( T_{2n+4}-T_{2n+3} \right) x_{-1}x_{-2} +T_{2n+3}x_{0}x_{-1}x_{-2} },
\end{equation*}
\end{corollary}

\begin{remark} When $a=d=0$, the system \eqref{system-generalized} takes the form

\begin{equation}\label{system_generalized_padovan}
   x_{n+1}=\dfrac{bx_{n-1}+c}{y_nx_{n-1}},\; y_{n+1}=\dfrac{by_{n-1}+c}{x_ny_{n-1}}
\; \; n\in \mathbb{N}_0.
\end{equation}
 As it is noted in \cite{arxiv}, the solutions are expressed using Padovan numbers. This system and same particular cases of it has been the subject of the papers \cite{Halim Rabago:2018}, \cite{Yazlik Tollu Taskara:2013}.

If $d=c=0$, The system \eqref{system-generalized} become

\begin{equation}\label{system_generalized_fibonaci}
   x_{n+1}=\dfrac{ay_{n}+b}{y_{n}},\; y_{n+1}=\dfrac{ax_{n}+b}{x_{n}},\, n\in \mathbb{N}_0.
\end{equation}
As it is noted in \cite{arxiv}:

- The  system \eqref{system_generalized_fibonaci} is a particular case of the more general system

\begin{equation}\label{ss}
 x_{n+1}=\dfrac{ay_{n}+b}{cy_{n}+d},\; y_{n+1}=\dfrac{\alpha x_{n}+\beta}{\gamma x_{n}+\lambda},\; n\in \mathbb{N}_0
\end{equation}
which was been completely solved by Stevic in \cite{stevicejqtde} and the solutions are expressed using a generalized Fibonacci sequence. 

- Also, particular cases of system \eqref{ss} has been studied in \cite{matsunaga}, \cite{Halim:2015}, \cite{Tollu Yazlik Taskara:2014(1)}, \cite{Tollu Yazlik Taskara:2013}.
\end{remark}

- We note that if also $b=0$, then the solutions of the system \eqref{system_generalized_fibonaci} are given by
$$\left\{\left(x_{0},y_{0}\right),\left(a,a\right),\left(a,a\right),...,\right\}.$$
\section{Proof of the Main Theorem}

In order to solve the system \eqref{system-generalized}, we need firstly to solve the following two homogeneous forth order linear difference equations

\begin{equation}\label{third-order_equation_R(n)}
R_{n+1}=aR_{n}+bR_{n-1}+cR_{n-2}+dR_{n-3},\,n\in \mathbb{N}_0,
\end{equation}

\begin{equation} \label{third-order_equation_S(n)}
S_{n+1}=-aS_{n}+bS_{n-1}-cS_{n-2}+dS_{n-3},\,n\in \mathbb{N}_0,
\end{equation}
where the initial values $R_{0}, R_{-1}$, $R_{-2}$, $R_{-3}$, $S_{0}, S_{-1}$, $S_{-2}$ and $S_{-3}$  and the constant coefficients $a$, $b$, $c$ and $d$ are real numbers with $d\neq 0$. In fact we will express the terms of the sequences $\left(R_{n}\right)_{n=-3}^{+\infty}$ and  $\left(S_{n}\right)_{n=-3}^{+\infty}$ using the sequence $\left(J_{n}\right)_{n=0}^{+\infty}$.\\

The difference equation \eqref{third-order_equation_R(n)} has the same characteristic equation as $\left(J_{n}\right)_{n=0}^{+\infty}$, that is the equation \eqref{eq_charac_of_GTS}. \\
To solve the difference equation \eqref{third-order_equation_S(n)} using terms of \eqref{tribonacci_sequence_J(n)}, we need the following fourth order linear  difference equation defined by
\begin{equation}\label{tribonacci_sequence_j(n)}
j_{n+4}=-aj_{n+3}+bj_{n+2}-cj_{n+1}+dj_{n},\quad n\in \mathbb{N}_0,
\end{equation}
and the special initial values
\begin{equation}
j_0=0, \quad j_1=0, \quad j_2=1 \; \text{and} \; j_3=-a
\end{equation}

The characteristic equation of \eqref{third-order_equation_S(n)} and \eqref{tribonacci_sequence_j(n)} is
   \begin{equation}
\lambda^4+a\lambda^3-b\lambda^2+c\lambda-d=0.   \label{eq_charac_of_GTS1}
\end{equation}
Clearly the roots of  \eqref{eq_charac_of_GTS1} are $-\alpha$, $-\beta$, $-\gamma$ and $-\delta$. Now following the same procedure in solving $\{J(n)\}$, it is not hard to see that
$$j(n)=(-1)^{n}J(n).$$

Now, we are able to prove the following result.
\begin{lemma} \label{lemma_R(n),S(n)_by_J(n)}
We have for all $n \in \mathbb{N}_0$,
\begin{equation}\label{R(n)_by_J(n)}
R_n= dJ_{n+1}R_{-3}+\left( cJ_{n+1}+dJ_n \right) R_{-2}+\left( J_{n+3}-aJ_{n+2} \right) R_{-1}+J_{n+2}R_0
\end{equation}
\begin{equation}\label{S(n)_by_J(n)}
S_n= (-1)^{n+1}\left[ dJ_{n+1}S_{-3}-\left( cJ_{n+1}+dJ_n \right) S_{-2}+\left( J_{n+3}-aJ_{n+2} \right) S_{-1}-J_{n+2}S_0 \right].
\end{equation}
\end{lemma}
\begin{proof}

Assume that $\alpha$, $ \beta$, $\gamma$  and $\delta$ are the distinct roots of the characteristic  equation  \eqref{eq_charac_of_GTS}, so
\begin{equation*}
R_{n}=c'_{1}\alpha^{n}+c'_{2}\beta^{n}+c'_{3}\gamma^{n}+c'_{4}\delta^n,\,n\geq -3.
\end{equation*}
Using the initial values $R_{0}, R_{-1}$, $R_{-2}$ and $R_{-3}$, we get
\begin{equation} \label{roots_and__c(i)_relations_of_R(n)}
\begin{cases}
\dfrac{1}{\alpha^3}c'_1 + \dfrac{1}{\beta^3}c'_2 + \dfrac{1}{\gamma^3}c'_3+ \dfrac{1}{\delta^3}c'_4 &=R_{-3} \\
\dfrac{1}{\alpha^2}c'_1 + \dfrac{1}{\beta^2}c'_2 + \dfrac{1}{\gamma^2}c'_3 +\dfrac{1}{\delta^2}c'_4 &=R_{-2} \\
\dfrac{1}{\alpha}c'_1 + \dfrac{1}{\beta}c'_2 + \dfrac{1}{\gamma}c'_3 + \dfrac{1}{\delta}c'_4 &=R_{-1} \\
c'_1+c'_2+c'_3+c'_4 &=R_0
\end{cases}
\end{equation}
after some  calculations using the Cramer method we get
\begin{eqnarray*}
c'_1 & = & \dfrac{\beta \gamma \delta \alpha^3}{(\delta-\alpha)(\gamma-\alpha)(\beta-\alpha)}R_{-3} - \dfrac{(\gamma\beta+\gamma\delta+\beta\delta)\alpha^3}{(\delta-\alpha)(\gamma-\alpha)(\beta-\alpha)} R_{-2}\\
         &&+  \dfrac{(\beta+\gamma+\delta)\alpha^3}{(\delta-\alpha)(\gamma-\alpha)(\beta-\alpha)}R_{-1}-\dfrac{\alpha^3}{(\delta-\alpha)(\gamma-\alpha)(\beta-\alpha)}R_0 \\
c'_2 & = & -\dfrac{\alpha \gamma \delta \beta^3}{(\delta-\beta)(\gamma-\beta)(\beta-\alpha)}R_{-3} + \dfrac{(\gamma\alpha+\gamma\delta+\alpha\delta)\beta^3}{(\delta-\beta)(\gamma-\beta)(\beta-\alpha)} R_{-2}\\
         && -  \dfrac{(\alpha+\gamma+\delta)\beta^3}{(\delta-\beta)(\gamma-\beta)(\beta-\alpha)}R_{-1}+\dfrac{\beta^3}{(\delta-\beta)(\gamma-\beta)(\beta-\alpha)}R_0 \\
c'_3 & = & \dfrac{\alpha\beta\delta \gamma^3}{(\delta-\gamma)(\gamma-\beta)(\gamma-\alpha)}R_{-3} - \dfrac{(\alpha\beta+\alpha\delta+\beta\delta)\gamma^3}{(\delta-\gamma)(\gamma-\beta)(\gamma-\alpha)} R_{-2}\\
         &&+ \dfrac{(\alpha+\beta+\delta)\gamma^3}{(\delta-\gamma)(\gamma-\beta)(\gamma-\alpha)} R_{-1}-\dfrac{\gamma^3}{(\delta-\gamma)(\gamma-\beta)(\gamma-\alpha)}R_0 \\
c'_4 & = &-\dfrac{\alpha\beta\gamma \delta^3}{(\delta-\gamma)(\delta-\beta)(\delta-\alpha)}R_{-3} + \dfrac{(\gamma\alpha+\gamma\beta+\alpha\beta)\delta^3}{(\delta-\gamma)(\delta-\beta)(\delta-\alpha)} R_{-2}\\
         && -  \dfrac{(\alpha+\beta+\gamma)\delta^3}{(\delta-\gamma)(\delta-\beta)(\delta-\alpha)} R_{-1}+\dfrac{\delta^3}{(\delta-\gamma)(\delta-\beta)(\delta-\alpha)}R_0
\end{eqnarray*}
that is,
\begin{eqnarray*}
R_n & = & \left(\dfrac{\beta \gamma \delta \alpha^3}{(\delta-\alpha)(\gamma-\alpha)(\beta-\alpha)}\alpha^n -\dfrac{\alpha \gamma \delta \beta^3}{(\delta-\beta)(\gamma-\beta)(\beta-\alpha)} \beta^n+\dfrac{\alpha\beta\delta \gamma^3}{(\delta-\gamma)(\gamma-\beta)(\gamma-\alpha)} \gamma^n \right. \\
&&\left. -\dfrac{\alpha\beta\gamma \delta^3}{(\delta-\gamma)(\delta-\beta)(\delta-\alpha)} \delta^n  \right) R_{-3} \\
 && +\: \left(- \dfrac{(\gamma\beta+\gamma\delta+\beta\delta)\alpha^3}{(\delta-\alpha)(\gamma-\alpha)(\beta-\alpha)} \alpha^n +\dfrac{(\gamma\alpha+\gamma\delta+\alpha\delta)\beta^3}{(\delta-\beta)(\gamma-\beta)(\beta-\alpha)}\beta^n - \dfrac{(\alpha\beta+\alpha\delta+\beta\delta)\gamma^3}{(\delta-\gamma)(\gamma-\beta)(\gamma-\alpha)} \gamma^n \right. \\
&&\left. + \dfrac{(\gamma\alpha+\gamma\beta+\alpha\beta)\delta^3}{(\delta-\gamma)(\delta-\beta)(\delta-\alpha)}\delta^n \right) R_{-2} \\
&& +\: \left( \dfrac{(\beta+\gamma+\delta)\alpha^3}{(\delta-\alpha)(\gamma-\alpha)(\beta-\alpha)}\alpha^n - \dfrac{(\alpha+\gamma+\delta)\beta^3}{(\delta-\beta)(\gamma-\beta)(\beta-\alpha)}\beta^n+ \dfrac{(\alpha+\beta+\delta)\gamma^3}{(\delta-\gamma)(\gamma-\beta)(\gamma-\alpha)}\gamma^n \right. \\
&&\left. - \dfrac{(\alpha+\beta+\gamma)\delta^3}{(\delta-\gamma)(\delta-\beta)(\delta-\alpha)} \delta^n \right) R_{-1} \\
&& +\: \left(-\dfrac{\alpha^3}{(\delta-\alpha)(\gamma-\alpha)(\beta-\alpha)}\alpha^n +\dfrac{\beta^3}{(\delta-\beta)(\gamma-\beta)(\beta-\alpha)}\beta^n -\dfrac{\gamma^3}{(\delta-\gamma)(\gamma-\beta)(\gamma-\alpha)}\gamma^n \right. \\
&&\left. +\dfrac{\delta^3}{(\delta-\gamma)(\delta-\beta)(\delta-\alpha)} \delta^n \right) R_{0}.
\end{eqnarray*}
\begin{equation*}
R_n= dJ_{n+1}R_{-3}+\left( cJ_{n+1}+dJ_n \right) R_{-2}+\left( J_{n+3}-aJ_{n+2} \right) R_{-1}+J_{n+2}R_0.
\end{equation*}
The proof of the other cases is similar and will be omitted.

Let $A:=-a$, $B:=b$, $C:=-c$ and $D:=d$ then equation \eqref{third-order_equation_S(n)} takes the form of \eqref{third-order_equation_R(n)} and the equation \eqref{tribonacci_sequence_j(n)} takes the form of \eqref{tribonacci_sequence_J(n)}. Then analogous to the formula of \eqref{third-order_equation_R(n)} we obtain  \begin{equation*}
S_n=  Dj_{n+1}S_{-3}+\left( Cj_{n+1}+Dj_n \right) S_{-2}+\left( j_{n+3}-Aj_{n+2} \right) S_{-1}+j_{n+2}S_0.
\end{equation*}
Using the fact that  $j(n)=(-1)^{n+1}J(n)$, $A=-a$ and $C:=-c$ we get
\begin{equation*}
S_n= (-1)^{n+1}\left[ dJ_{n+1}S_{-3}-\left( cJ_{n+1}+dJ_n \right) S_{-2}+\left( J_{n+3}-aJ_{n+2} \right) S_{-1}-J_{n+2}S_0 \right].
\end{equation*}
\end{proof}

\emph{\textbf{Proof of the Main Theorem.}}

Putting
\begin{equation} \label{changing_variable_x(n)_y(n)_by_u(n)_v(n)}
x_{n}=\dfrac{u_{n}}{v_{n-1}},\quad y_{n}=\dfrac{v_{n}}{u_{n-1}},\,n \geq -2.
\end{equation}
we get the following linear forth order system of  difference
equations
\begin{equation}\label{sys_linear}
u_{n+1}=av_{n}+bu_{n-1}+cv_{n-2}+du_{n-3},\quad
v_{n+1}=au_{n}+bv_{n-1}+cu_{n-2}+dv_{n-3},\quad n \in \mathbb{N}_0,
\end{equation}
where the initial values $u_{-3},u_{-2},u_{-1},u_0, v_{-3}, v_{-2},v_{-1},v_0$ are nonzero real numbers.\\
From\eqref{sys_linear} we have for $n \in \mathbb{N}_0$,
\begin{equation*}
\begin{cases}
u_{n+1}+v_{n+1}=a(v_{n}+u_{n})+b(u_{n-1}+v_{n-1})+c(v_{n-2}+u_{n-2})+d(u_{n-3}+v_{n-3}),\\
u_{n+1}-v_{n+1}=a(v_{n}-u_{n})+b(u_{n-1}-v_{n-1})+c(v_{n-2}-u_{n-2})+d(u_{n-3}-v_{n-3}).
\end{cases}
\end{equation*}
Putting again
\begin{equation}\label{u,v}
R_n=u_n+v_n, \quad S_n=u_n-v_n,\,n\geq -2,
\end{equation}\label{R,S}
we obtain two homogeneous linear difference equations of forth order:
\begin{equation*}
R_{n+1}=aR_{n}+bR_{n-1}+cR_{n-2}+dR_{n-3},\, n \in \mathbb{N}_0,
\end{equation*}
and
\begin{equation} \label{S(n)_equation}
S_{n+1}=-aS_{n}+bS_{n-1}-cS_{n-2}+dS_{n-3},\, n \in \mathbb{N}_0.
\end{equation}

Using \eqref{u,v}, we get for $n \geq -3$,
\begin{equation*}
u_n= \dfrac{1}{2}(R_n+S_n), \; v_n= \dfrac{1}{2}(R_n-S_n).
\end{equation*}
From Lemma \ref{lemma_R(n),S(n)_by_J(n)} we obtain,

\begin{eqnarray} \label{solution_form_of_u(n)}
\begin{cases}
u_{2n-1}=&\dfrac{1}{2}\left[ dJ_{2n}(R_{-3}+S_{-3})+\left( cJ_{2n}+dJ_{2n-1} \right) (R_{-2}-S_{-2}) +\left( J_{2n+2}-aJ_{2n+1} \right) (R_{-1}+S_{-1}) \right. \\
                &\left. +J_{2n+1}(R_0-S_0)  \right], n \in \mathbb{N},\\
u_{2n}=&\dfrac{1}{2}\left[dJ_{2n+1}(R_{-3}-S_{-3})+\left( cJ_{2n+1}+dJ_{2n} \right) (R_{-2}+S_{-2}) +\left( J_{2n+3}-aJ_{2n+2} \right) (R_{-1}-S_{-1}) \right. \\
                &\left. +J_{2n+2}(R_0+S_0)  \right], n \in \mathbb{N}_0,
\end{cases}
\end{eqnarray}
\begin{eqnarray}
\begin{cases} \label{solution_form_of_v(n)}
v_{2n-1}=&\dfrac{1}{2}\left[ dJ_{2n}(R_{-3}-S_{-3})+\left( cJ_{2n}+dJ_{2n-1} \right) (R_{-2}+S_{-2}) +\left( J_{2n+2}-aJ_{2n+1} \right) (R_{-1}-S_{-1}) \right. \\
                &\left. +J_{2n+1}(R_0+S_0)  \right], n \in \mathbb{N},\\
v_{2n}=&\dfrac{1}{2}\left[dJ_{2n+1}(R_{-3}+S_{-3})+\left( cJ_{2n+1}+dJ_{2n} \right) (R_{-2}-S_{-2}) +\left( J_{2n+3}-aJ_{2n+2} \right) (R_{-1}+S_{-1}) \right. \\
                &\left. +J_{2n+2}(R_0-S_0)  \right], n \in \mathbb{N}_0,
\end{cases}
\end{eqnarray}
Substituting \eqref{solution_form_of_u(n)} and \eqref{solution_form_of_v(n)} in \eqref{changing_variable_x(n)_y(n)_by_u(n)_v(n)}, we get for $n \in \mathbb{N}_0$,

\begin{equation}\label{x(na)}
x_{2n+1}=\dfrac{dJ_{2n+2}+\left( cJ_{2n+2}+dJ_{2n+1} \right) \dfrac{R_{-2}-S_{-2}}{R_{-3}+S_{-3}} +\left( J_{2n+4}-aJ_{2n+3} \right) \dfrac{R_{-1}+S_{-1}}{R_{-3}+S_{-3}}+J_{2n+3} \dfrac{R_0-S_0}{R_{-3}+S_{-3}}}{dJ_{2n+1} +\left( cJ_{2n+1}+dJ_{2n} \right) \dfrac{R_{-2}-S_{-2}}{R_{-3}+S_{-3}} +\left( J_{2n+3}-aJ_{2n+2} \right) \dfrac{R_{-1}+S_{-1}}{R_{-3}+S_{-3}}+J_{2n+2} \dfrac{R_0-S_0}{R_{-3}+S_{-3}}},
\end{equation}

\begin{equation}\label{x(nb)}
x_{2n+2}=\dfrac{dJ_{2n+3}+\left( cJ_{2n+3}+dJ_{2n+2} \right) \dfrac{R_{-2}+S_{-2}}{R_{-3}-S_{-3}} +\left( J_{2n+5}-aJ_{2n+4} \right) \dfrac{R_{-1}-S_{-1}}{R_{-3}-S_{-3}}+J_{2n+4}\dfrac{R_0+S_0}{R_{-3}-S_{-3}}}{dJ_{2n+2}+\left( cJ_{2n+2}+dJ_{2n+1} \right) \dfrac{R_{-2}+S_{-2}}{R_{-3}-S_{-3}} +\left( J_{2n+4}-aJ_{2n+3} \right) \dfrac{R_{-1}-S_{-1}}{R_{-3}-S_{-3}}+J_{2n+3} \dfrac{R_0+S_0}{R_{-3}-S_{-3}}},
\end{equation}

\begin{equation}\label{y(na)}
y_{2n+1}=\dfrac{dJ_{2n+2}+\left( cJ_{2n+2}+dJ_{2n+1} \right) \dfrac{R_{-2}+S_{-2}}{R_{-3}-S_{-3}} +\left( J_{2n+4}-aJ_{2n+3} \right) \dfrac{R_{-1}-S_{-1}}{R_{-3}-S_{-3}}+J_{2n+3} \dfrac{R_0+S_0}{R_{-3}-S_{-3}}}{dJ_{2n+1}+\left( cJ_{2n+1}+dJ_{2n} \right) \dfrac{R_{-2}+S_{-2}}{R_{-3}-S_{-3}} +\left( J_{2n+3}-aJ_{2n+2} \right) \dfrac{R_{-1}-S_{-1}}{R_{-3}-S_{-3}}+J_{2n+2} \dfrac{R_0+S_0}{R_{-3}-S_{-3}}},
\end{equation}
and
\begin{equation}\label{y(nb)}
y_{2n+2}=\dfrac{dJ_{2n+3}+\left( cJ_{2n+3}+dJ_{2n+2} \right) \dfrac{R_{-2}-S_{-2}}{R_{-3}+S_{-3}} +\left( J_{2n+5}-aJ_{2n+4} \right) \dfrac{R_{-1}+S_{-1}}{R_{-3}+S_{-3}}+J_{2n+4}\dfrac{R_0-S_0}{R_{-3}+S_{-3}}}{dJ_{2n+2}+\left( cJ_{2n+2}+dJ_{2n+1} \right) \dfrac{R_{-2}-S_{-2}}{R_{-3}+S_{-3}} +\left( J_{2n+4}-aJ_{2n+3} \right) \dfrac{R_{-1}+S_{-1}}{R_{-3}+S_{-3}}+J_{2n+3} \dfrac{R_0-S_0}{R_{-3}+S_{-3}}}.
\end{equation}

We have
\begin{equation} \label{xrs}
x_{-2}=\dfrac{u_{-2}}{v_{-3}}=\dfrac{R_{-2}+S_{-2}}{R_{-3}-S_{-3}},\,  x_{-1}=\dfrac{u_{-1}}{v_{-2}}=\dfrac{R_{-1}+S_{-1}}{R_{-2}-S_{-2}},\,x_0=\dfrac{u_0}{v_{-1}}=\dfrac{R_{0}+S_{0}}{R_{-1}-S_{-1}},
  \end{equation}
  \begin{equation} \label{yrs}
y_{-2}=\dfrac{v_{-2}}{u_{-3}}=\dfrac{R_{-2}-S_{-2}}{R_{-3}+S_{-3}},\,  y_{-1}=\dfrac{v_{-1}}{u_{-2}}=\dfrac{R_{-1}-S_{-1}}{R_{-2}+S_{-2}},\, y_0=\dfrac{v_0}{u_{-1}}=\dfrac{R_{0}-S_{0}}{R_{-1}+S_{-1}}
\end{equation}
From \eqref{xrs}, \eqref{yrs} we get,

\begin{equation}\label{xyrs1}
\begin{cases}
\dfrac{R_{-1}+S_{-1}}{R_{-3}+S_{-3}}=\dfrac{R_{-1}+S_{-1}}{R_{-2}-S_{-2}} \times \dfrac{R_{-2}-S_{-2}}{R_{-3}+S_{-3}}= x_{-1}y_{-2} \\
\dfrac{R_{0}-S_{0}}{R_{-3}+S_{-3}}= \dfrac{R_{0}-S_{0}}{R_{-1}+S_{-1}} \times \dfrac{R_{-1}+S_{-1}}{R_{-2}-S_{-2}} \times \dfrac{R_{-2}-S_{-2}}{R_{-3}+S_{-3}} =y_0x_{-1}y_{-2}
\end{cases}
\end{equation}
\begin{equation}\label{xyrs2}
\begin{cases}
\dfrac{R_{-1}-S_{-1}}{R_{-3}-S_{-3}}=\dfrac{R_{-1}-S_{-1}}{R_{-2}+S_{-2}} \times \dfrac{R_{-2}+S_{-2}}{R_{-3}-S_{-3}}= y_{-1}x_{-2} \\
\dfrac{R_{0}+S_{0}}{R_{-3}-S_{-3}}= \dfrac{R_{0}+S_{0}}{R_{-1}-S_{-1}} \times \dfrac{R_{-1}-S_{-1}}{R_{-2}+S_{-2}} \times \dfrac{R_{-2}+S_{-2}}{R_{-3}-S_{-3}} =x_0y_{-1}x_{-2}
\end{cases}
\end{equation}
Using \eqref{x(na)}, \eqref{x(nb)}, \eqref{y(na)}, \eqref{y(nb)}, \eqref{xyrs1} and \eqref{xyrs2}, we obtain the closed form of the solutions of the system \eqref{system-generalized}, that is for $n \in \mathbb{N}_0$, we have
\begin{eqnarray*}
\begin{cases}
x_{2n+1}=&\dfrac{dJ_{2n+2}+\left( cJ_{2n+2}+dJ_{2n+1} \right) y_{-2} +\left( J_{2n+4}-aJ_{2n+3} \right) x_{-1}y_{-2} +J_{2n+3} y_{0}x_{-1}y_{-2} }{ dJ_{2n+1} +\left( cJ_{2n+1}+dJ_{2n} \right) y_{-2} +\left( J_{2n+3}-aJ_{2n+2} \right) x_{-1}y_{-2}+J_{2n+2} y_{0}x_{-1}y_{-2}} ,\\
x_{2n+2}=&\dfrac{dJ_{2n+3}+\left( cJ_{2n+3}+dJ_{2n+2} \right) x_{-2} +\left( J_{2n+5}-aJ_{2n+4} \right) y_{-1}x_{-2} +J_{2n+4}x_{0}y_{-1}x_{-2}}{dJ_{2n+2}+\left( cJ_{2n+2}+dJ_{2n+1} \right) x_{-2} +\left( J_{2n+4}-aJ_{2n+3} \right) y_{-1}x_{-2} +J_{2n+3}x_{0}y_{-1}x_{-2} },
\end{cases}
\end{eqnarray*}
\begin{eqnarray*}
\begin{cases}
y_{2n+1}=&\dfrac{dJ_{2n+2}+\left( cJ_{2n+2}+dJ_{2n+1} \right) x_{-2} +\left( J_{2n+4}-aJ_{2n+3} \right) y_{-1}x_{-2} +J_{2n+3}x_{0}y_{-1}x_{-2}}{dJ_{2n+1}+\left( cJ_{2n+1}+dJ_{2n} \right) x_{-2} +\left( J_{2n+3}-aJ_{2n+2} \right) y_{-1}x_{-2}+J_{2n+2} x_{0}y_{-1}x_{-2} },\\
y_{2n+2}=&\dfrac{dJ_{2n+3}+\left( cJ_{2n+3}+dJ_{2n+2} \right) y_{-2} +\left( J_{2n+5}-aJ_{2n+4} \right) x_{-1}y_{-2} +J_{2n+4} y_{0}x_{-1}y_{-2}}{dJ_{2n+2}+\left( cJ_{2n+2}+dJ_{2n+1} \right) y_{-2} +\left( J_{2n+4}-aJ_{2n+3} \right) x_{-1}y_{-2}+J_{2n+3} y_{0}x_{-1}y_{-2}}.
\end{cases}
\end{eqnarray*}

%
%
\bibliography{mmnsample}

\begin{thebibliography}{99}
\bibitem{arxiv}  Y. Akrour, N. Touafek, Y. Halim, \emph{On a system of difference equations of second order solved in  closed form}, arXiv:1904.04476, 2019.
\bibitem{azizi}  R. Azizi, \emph{Global behaviour of the rational Riccati difference equation of order two: the general case}, J. Difference Equ. Appl., vol. 18, no. 6, pp. 947-961, 2012.

\bibitem{azizi1}  R. Azizi, \emph{Global behavior of the higher order rational Riccati difference equation}, Appl. Math. Comput. 230, pp. 1-8, 2014.
\bibitem{Elsayed:2015}
E. M. Elsayed,  \emph{On a system of two nonlinear difference equations of order two}, Proc. Jangeon Math. Soc., vol. 18, no. 1, pp. 353-368, 2015.
\bibitem{Elsayed Ibrahim:2015}
E. M. Elsayed, T. F. Ibrahim,\emph{Periodicity and solutions for some systems of nonlinear rational difference equations}, Hacet. J. Math. Stat., vol. 44, no. 1, pp. 1361-1390, 2015.
\bibitem{Elsayed:2014}
E. M. Elsayed, \emph{Solution for systems of difference equations of rational form of order two}, Comp. Appl. Math.,  vol. 33, no. 1, pp. 751-765, 2014.
\bibitem{gumus} M. Gumus, R. Abo-Zeid, \emph{On the solutions of a (2k+2)th order difference equation}. Dyn. Contin. Discrete Impuls Syst. Ser. B Appl. Algorithms vol. 25, pp. 129"1¤7143,  2018.
\bibitem{Halim Rabago:2018}
Y. Halim, J. F. T. Rabago,    \emph{On the solutions of a second-order difference equation in terms of generalized Padovan sequences}. 	Math. Slovaca, vol. 68, no. 3, pp. 625-638, 2018.
\bibitem{Halim Rabago:2017}
Y. Halim Y, J. F. T. Rabago,\emph{On some solvable systems of difference equations with solutions associated to Fibonacci numbers}. Electron J. Math. Analysis Appl,
vol. 5, no. 1, pp. 166-178, 2017.
\bibitem{Halim:2016}
Y. Halim, \emph{A system of difference equations with solutions associated to Fibonacci numbers}. Int. J. Difference Equ., vol. 11, no. 1, pp. 65-77, 2016.
\bibitem{Halim Bayram:2016}
Y. Halim, M. Bayram, \emph{On the solutions of a higher-order difference equation in terms of generalized Fibonacci sequences}.
	Math. Methods Appl. Sci., vol. 39, no. 1, pp. 2974-2982, 2016.
\bibitem{Halim Touafek:2014}
Y. Halim, N. Touafek N, E. M. Elsayed, \emph{Closed form solution of some systems of rational
difference equations in terms of Fibonacci numbers}. Dyn. Contin. Discrete Impuls. Syst., Ser. A, Math. Anal., vol. 21, no. 6, pp. 473-486, 2014.
\bibitem{Halim:2015}
 Y. Halim, \emph{Global character of systems of rational difference equations}, 	Electron. J. Math. Analysis Appl.,
vol. 3, no. 1, pp. 204-214, 2015.
\bibitem{matsunaga}  H. Matsunaga, R. Suzuki, \emph{Classification of global behavior of a system of rational difference equations},
 Appl. Math. Lett., vol. 85, no. 1, pp. 57--63, 2018.
\bibitem{stevicejqtde}  S. Stevic, \emph{Representation of solutions of bilinear difference equations in terms of generalized Fibonacci sequences}, Electron. J. Qual. Theory Differ. Equ., vol. 67, no. 1, pp. 15 pages, 2014.
\bibitem{stevic2004}  S. Stevic, \emph{More on a rational recurrence relation}, Appl. Math. E-Notes, vol. 4, no. 1, pp. 80-85, 2004.
\bibitem{stevicejqtde2018}  S. Stevic, \emph{Representation of solutions of a solvable nonlinear
difference equation of second order}, Electron. J. Qual. Theory Differ. Equ., vol. 95, no. 1, pp. 18 pages, 2018.
\bibitem{Tollu Yazlik Taskara:2013}
D. T. Tollu, Y. Yazlik, N. Taskara, \emph{On the solutions of two special types of Riccati difference equation via Fibonacci umbers}, vol. 174, no. 1, pp. 7 pages, 2013.
\bibitem{Tollu Yazlik Taskara:2014(1)}
D. T. Tollu, Y. Yazlik, N. Taskara,  \emph{The solutions of four Riccati difference equations associated with Fibonacci numbers},  Balkan J. Math.
vol. 2, no. 1, pp. 163-172, 2014.
\bibitem{Tollu Yazlik Taskara:2014(2)}
D. T. Tollu, Y. Yazlik, N. Taskara,  \emph{On fourteen solvable systems of difference equations}, Appl. Math. $\&$ Comp.,  vol. 233, no. 1, pp. 310-319, 2014.
\bibitem{Touafek:2014}
N. Touafek, \emph{On some fractional systems of difference equations}, Iran. J. Math. Sci. Inform.,
 vol. 9, no. 2, pp. 303-305, 2014.

\bibitem{Touafek Halim:2011}
N. Touafek, Y. Halim, \emph{On max type difference equations: expressions of solutions}, Int. J. Appl. Nonlinear Sci.,
 vol. 11, no. 4, pp. 396-402, 2011.
\bibitem{Touafek Elsayed:2012(1)}
N. Touafek, E. M. Elsayed, \emph{On the periodicity of some systems of nonlinear difference equations},
 Bull. Math. Soc. Sci. Math. Roum., Nouv. Ser.,  vol. 55, no. 1, pp. 217-224, 2012.

\bibitem{Touafek Elsayed:2012(2)}
N. Touafek, E. M. Elsayed, \emph{On the solutions of systems of rational difference equations},  	Math. Comput. Modelling,
 vol. 55, no. 1, pp. 1987-1997, 2012.
\bibitem{wang} C. Wang, X. Fang, R. Li \emph{On the solution for a system of two rational difference equations}, J. Comput. Anal. Appl. vol. 20, no. 1, pp. 175"1¤7186, 2016.
\bibitem{Yazlik Tollu Taskara:2013}
Y. Yazlik, D. T. Tollu, N. Taskara, \emph{On the solutions of difference equation systems with Padovan numbers},  Appl Math.,
  vol. 12, no. 1, pp. 15-20, 2013.
\bibitem{waddill1992tetranacci}
Waddill, Marcellus E, \emph{The Tetranacci sequence and generalizations},The Fibonacci Quarterly, vol 30, N, pp 9--20, 1992, Citeseer.
\bibitem{hathiwala2017binet} Hathiwala, Gautam S and Shah, Devbhadra V,  \emph{Binet--type formula for the sequence of Tetranacci numbers by alternate methods}, Mathematical Journal of Interdisciplinary Sciences, vol 6, N 1, pp 37--48, 2017.



\end{thebibliography}
\bibliographystyle{mmn}

\end{document}